\documentclass[11pt]{article}
\usepackage{soul}
\usepackage{exscale,relsize}
\usepackage{amsmath}
\usepackage{amsfonts}
\usepackage[hidelinks]{hyperref}
\usepackage{amssymb}
\usepackage{calc}
\usepackage{theorem}
\usepackage{pifont}      
\usepackage{array}
\usepackage{color}
\usepackage{graphicx}
\usepackage{comment}
\usepackage[makeroom]{cancel}
 \usepackage{mathpazo}
\usepackage{enumitem}
\usepackage{todonotes}

\newcommand{\sumop}{\ensuremath{{S}}}
\newcommand{\To}{\ensuremath{\rightrightarrows}}

\newcommand{\scal}[2]{\langle{{#1},{#2}}\rangle}
\newcommand{\Scal}[2]{\left\langle{{#1},{#2}}\right\rangle}

\newcommand{\ST}{\ensuremath{\;|\;}}

\newcommand{\RR}{\ensuremath{\mathbb R}}

\newcommand{\RX}{\ensuremath{\,\left]-\infty,+\infty\right]}}

\newcommand{\NN}{\ensuremath{\mathbb N}}

\newcommand{\menge}[2]{\big\{{#1} \mid {#2}\big\}}

\newcommand{\dom}{\ensuremath{\operatorname{dom}}}

\newcommand{\gra}{\ensuremath{\operatorname{gra}}}

\newcommand{\prox}{\ensuremath{\operatorname{Prox}}}

\newcommand{\ran}{\ensuremath{\operatorname{ran}}}

\newcommand{\spa}{\ensuremath{\operatorname{span}}}

\newcommand{\Id}{\ensuremath{\operatorname{Id}}}

\newcommand{\gr}{\ensuremath{\operatorname{gra}}}

\renewcommand{\iff}{\ensuremath{\Leftrightarrow}}
\renewcommand{\phi}{\ensuremath{\varphi}}

\newtheorem{theorem}{Theorem}[section]
\newtheorem{lemma}[theorem]{Lemma}
\newtheorem{fact}[theorem]{Fact}
\newtheorem{corollary}[theorem]{Corollary}
\newtheorem{proposition}[theorem]{Proposition}
\newtheorem{definition}[theorem]{Definition}

\theoremstyle{plain}{\theorembodyfont{\rmfamily}
}
\theoremstyle{plain}{\theorembodyfont{\rmfamily}
}
\theoremstyle{plain}{\theorembodyfont{\rmfamily}
}
\theoremstyle{plain}{\theorembodyfont{\rmfamily}
\newtheorem{example}[theorem]{Example}}
\theoremstyle{plain}{\theorembodyfont{\rmfamily}
\newtheorem{remark}[theorem]{Remark}}
\theoremstyle{plain}{\theorembodyfont{\rmfamily}
}

\newcommand{\pluss}{{\hskip1pt \raise1pt\vbox{\hrule width6pt \vskip1pt
\hrule width6pt}\kern-4pt{\lower1pt\hbox{\vrule height6pt \kern1pt\vrule
height6pt}}\hskip5pt}}

\newcommand{\argmin}{\mathop{\rm argmin}\limits}

\begin{document}

\title{\textsc
	Multi-marginal maximal monotonicity\\ and convex analysis}

\author{
	Sedi Bartz\thanks{
		Mathematics, University
		of Massachusetts Lowell, MA 01854, USA. E-mail:
		\texttt{sedi\_bartz@uml.edu}.},~~
	Heinz H.\ Bauschke\thanks{
		Mathematics, University
		of British Columbia,
		Kelowna, B.C.\ V1V~1V7, Canada. E-mail:
		\texttt{heinz.bauschke@ubc.ca}.},~~
	Hung M.\ Phan\thanks{
		Mathematics, University
		of Massachusetts Lowell, MA 01854, USA. E-mail:
		\texttt{hung\_phan@uml.edu}.},
	~~and Xianfu Wang\thanks{
		Mathematics, University of British Columbia, Kelowna, B.C.\
		V1V~1V7, Canada.
		E-mail: \texttt{shawn.wang@ubc.ca}.}}

\date{September 12, 2019}

\maketitle

\begin{abstract}
 Monotonicity and convex analysis arise naturally in the framework of
 multi-marginal optimal transport theory. However, a comprehensive
 multi-marginal monotonicity and convex analysis theory is still missing. 
 To this end we study extensions of classical monotone operator theory and
 convex analysis into the multi-marginal setting. We characterize
 multi-marginal $c$-monotonicity in terms of classical monotonicity and
 firmly nonexpansive mappings. We provide Minty type, continuity and
 conjugacy criteria for multi-marginal maximal monotonicity. We extend the
 partition of the identity into a sum of firmly nonexpansive mappings and
 Moreau's decomposition of the quadratic function into envelopes and proximal
 mappings into the multi-marginal settings. We illustrate our discussion with
 examples and provide applications for the determination of multi-marginal
 maximal monotonicity and multi-marginal conjugacy. We also point out several
 open questions.
\end{abstract}
{\small
	
}

\section{Introduction}

Our discussion stems from multi-marginal optimal transport theory: Let $(X_1,\mu_1),\ldots,(X_N,\mu_N)$ be
Borel probability spaces. We set $X=X_1\times\cdots\times X_N$ and we
denote by $\Pi(X)$ the set of all Borel probability measures $\pi$ on
$X$ such that the {\em marginals} of $\pi$ are the $\mu_i$'s. Let
$c:X\to\RR$ be a cost function. A cornerstone of multi-marginal
optimal transport theory is Kellerer's~\cite{Kel} generalization
of the Kantorovich duality theorem to the multi-marginal case. Kellerer's duality
theorem asserts that, in a suitable framework,     
\begin{equation}\label{Kellerer}
\min_{\pi\in\Pi(X)} \int_X c(x)d\pi(x)=\max_{\begin{array}{c}
	u_i\in L_1({\mu_i}),\\ 
	\sum_{1\leq i\leq N}u_i\ \leq\ c	
	\end{array}} \ \ \sum_{1\leq i\leq N}\int_{X_i}u_i(x_i)d\mu_i(x_i).
\end{equation}
It follows that  if $\pi$ is a solution of the left-hand side
of~\eqref{Kellerer} and $(u_1,\ldots,u_N)$ is a solution of the
right-hand side of~\eqref{Kellerer}, then $\pi$ is concentrated
on the subset $\Gamma$ of $X$ where the equality $c=\sum_{1\leq
	i\leq N}u_i$ holds. In recent publications (see, for example,~\cite{BBW2, Gri, KP}) such subsets $\Gamma$ of $X$ are referred to as $c$-splitting
sets: Let $N\geq 2$ be a natural number and $I=\{1,\ldots,N\}$ an index set. Let $X_1,\ldots,X_N$ be nonempty sets, $X=X_1\times\cdots\times X_N$ and $c:X\to\RR$ a function. 
\begin{definition}[$c$-splitting set]
\label{splitting def}
	Let $\Gamma\subseteq X$. We say that $\Gamma$ is a $c$-splitting set if for each $i\in I$ there exists a function $u_i:X_i\to\RX$ such that
	\begin{equation}\label{splitting functions definition inequality}
	\forall x=(x_1,\ldots,x_N)\in X,\qquad
	c(x_1,\ldots,x_N)\leq \left(\bigoplus_{i\in I}u_i\right)(x):=\sum_{i\in I} u_i(x_i)
	\end{equation}
	and 
	\begin{equation}\label{splitting functions definition equality}
	\forall x=(x_1,\ldots,x_N)\in \Gamma,\ \ \qquad
	c(x_1,\ldots,x_N)=\left(\bigoplus_{i\in I} u_i\right)(x):=\sum_{i\in I} u_i(x_i).
	\end{equation}
	In this case we say that $(u_1,\ldots,u_N)$ is a $c$-splitting tuple of $\Gamma$. Given functions $u_i:X_i\to\RX$ that satisfy~\eqref{splitting functions definition inequality}, we call the set of all points $(x_1,\ldots,x_N)\in X$ that satisfy \eqref{splitting functions definition equality} the $c$-splitting set generated by the tuple $(u_1,\ldots,u_N)$.
\end{definition}

 In the case $N=2$, splitting sets are natural in convex
analysis as graphs of subdifferentials. Indeed, by the Young-Fenchel
inequality the graph of the subdifferential $\partial f$ is the $c$-splitting
set generated by the pair $(f,f^*)$ where $c=\scal{\cdot}{\cdot}$ is the
classical pairing between a linear space and its dual. Similar to the
two-marginal case, in the multi-marginal case monotonicity arises naturally
as well:

\begin{definition}
[$c$-cyclic monotonicity]
\label{monotonicity definitions}
	The subset $\Gamma$ of $X$ is said to be $c$-cyclically monotone
	of order $n$, $n$-$c$-monotone for short, if for all $n$ tuples
	$(x^1_1,\dots,x_N^1),\dots,(x_1^n,\dots,x_N^n)$ in $\Gamma$ and
	every $N$ permutations $\sigma_1,\dots,\sigma_N$ in $S_n$,
	\begin{equation}\label{cycmondef}
	\sum_{j=1}^nc(x_1^{\sigma_1(j)},\dots,x_N^{\sigma_N(j)})\leq \sum_{j=1}^n c(x_1^{j},\dots,x_N^{j});
	\end{equation}
	$\Gamma$ is said to be $c$-cyclically monotone if it is
	$n$-$c$-monotone for every $n\in\{2,3,\dots\}$; and 
	$\Gamma$ is said to be $c$-monotone if it is $2$-$c$-monotone. Finally, $\Gamma$ is said to be maximally $n$-$c$-monotone if it has no proper $n$-$c$-monotone extension.
\end{definition}

Cyclic monotonicity was first introduced by Rockafellar~\cite{Rockafellar} in the framework of classical convex analysis. During the late 80s and early 90s (see  \cite{Bre, Rochet, Rus}) the concept was generalized to $c$-cyclic monotonicity in order to hold for more general cost functions $c$ in the framework of two-marginal optimal transport theory.  Currently, it lays at the foundations of the theory (see for example \cite{GanMc, San, Vil}) and plays a role also in recent refinements (see, for example, \cite{BR1, BR2}). Extending the role it plays in two-marginal optimal transport theory, in the past two and a half decades multi-marginal $c$-monotonicity and aspects of $c$-convex analysis are becoming an integral part of the fast evolving multi-marginal optimal transport theory as can be seen, for example, in ~\cite{AC, BBW2, BG, Car, CN, GS, GhMa, GhMo, Gri, KP, KS, MPC, MGN, Pas1, Pas2, RU}. An important instance of an extension  from the two-marginal case relating Definition~\ref{splitting def} with Definition~\ref{monotonicity definitions} is the known fact that $c$-splitting
sets are $c$-cyclically monotone (see, for example, \cite{BBW2, Gri, KP, KS}). 

Before attending our convex analytic discussion we remark that in order to make optimal transport compatible with our discussion, one should exchange min for max in the left-hand side of~\eqref{Kellerer}, exchange max for min in the right-hand side of~\eqref{Kellerer} and, finally, exchange the constraint $\sum_i u_i\leq c$ in the right-hand side of~\eqref{Kellerer} with the constraint $c\leq \sum_i u_i$ as we did in Definition~\ref{splitting def} and Definition~\ref{monotonicity definitions}.

In the framework of multi-marginal optimal transport, presumably the most
traditional and well studied cost functions are classical extensions of the
pairing between a linear space and its dual: 
\begin{quote}
	\emph{For the remainder of our
	discussion, for each $1\leq i\leq N$, we assume that
	$X_i=H$ is a real Hilbert space with inner product $\scal{\cdot}{\cdot}$ and induced norm $\|\cdot \|$.
	We let $c:X\to\RR$ be the cost function defined by}
	\begin{equation*}\label{classical c}
	c(x_1,\dots,x_N)=\sum_{1\leq i<j\leq N}\scal{x_i}{x_j}.
	\end{equation*}
\end{quote}
It follows from straightforward computation (see for example~\cite{BBW2})
that a set $\Gamma\subseteq X$ is $n$-$c$-monotone if and only if it is
$n$-$c$-monotone with respect to each of the functions
	\begin{equation*}
(x_1,\dots,x_N)\mapsto\ -\sum_{1\leq i<j\leq N}\tfrac{1}{2}\|x_i-x_j\|^2\ \ \ \ \ \ \ \ \ \ \ \ \text{and}\ \ \ \ \ \ \ \ \ \ \ \ \ \ (x_1,\dots,x_N)\mapsto\ \tfrac{1}{2}\bigg\|\sum_{i=1}^N x_i\bigg\|^2.
	\end{equation*}
Although classical convex analysis and monotonicity are instrumental in
multi-marginal optimal transport, and although several multi-marginal convex
analytic results are already available (as we recall in our more specific
discussion further below), to the best of our knowledge, a comprehensive
multi-marginal monotonicity and convex analysis theory is still lacking. 
To this end, in the present paper we lay additional foundations and provide
several extensions of classical monotone operator theory and convex analysis
into the multi-marginal settings.

The remainder of the paper is organized as follows. In
Section~\ref{s:multi-mar-Minty} we provide a characterization of
multi-marginal $c$-monotonicity in terms of classical monotonicity. We employ
this characterization in order to provide several equivalent criteria,
including a Minty-type criterion, a criterion based on the
partition of the identity into a sum of
firmly nonexpansive mappings, and other criteria for multi-marginal
maximal $c$-monotonicity. In Section~\ref{s:MMMM_via_cont} we provide a
continuity criterion for multi-marginal maximal monotonicity. In
Section~\ref{s:c-split_tuples} we focus on multi-marginal convex analysis. In
particular, we extend Moreau's decompositions and provide criteria for
maximal $c$-monotonicity of $c$-splitting sets, the multi-marginal extensions of
subdifferentials. We show that the same criteria also imply multi-marginal
$c$-conjugacy of $c$-splitting functions. In the case $N=3$ we also provide a
class of $c$-splitting triples for which $c$-conjugacy implies maximal
$c$-monotonicity. Section~\ref{s:ex} contains examples and applications of
our results to the problem of determining maximal $c$-monotonicity of sets
and $c$-conjugacy of $c$-splitting tuples, thus reducing the need of further
challenging computations of multi-marginal $c$-conjugate tuples.
Additionally, we point out several open problems.

In the remainder of this section we collect standard notations and
preliminary facts from classical monotone operator theory and convex analysis
which, largely, follow \cite{BC2017}. Let $A:H \To H$ be a set-valued
mapping. The {\em domain} of $A$ is the set $\dom A=\{x\in H \ST
Ax\neq\varnothing\}$. The \emph{range} of $A$ is the set $\ran A =
A(H)=\bigcup_{x \in H} Ax$, the \emph{graph} of $A$ is the set $\gra A =
\{(x,u)\in H \times H \ST u \in Ax\}$ and the {\em inverse mapping} of
$A$ is the mapping $A^{-1}$ satisfying $x\in A^{-1}u\Leftrightarrow u\in Ax$.
$A$ is said to be \emph{monotone} if
$$
(\forall (x,u) \in \gra A)(\forall (y,v) \in \gra A)\quad \scal{x-y}{u-v} \geq 0.
$$
$A$ is said to be \emph{maximally monotone} if there exists no monotone
operator $B$ such that $\gra A$ is a proper subset of $\gra B$. The
\emph{resolvent} of $A$ is the mapping $J_A=(A+\Id)^{-1}$ where $\Id$ is the
{\em identity mapping}.
The mapping $T:\dom T\subseteq H\to H$ is said to be \emph{firmly nonexpansive} if
$$
(\forall x\in \dom T)(\forall y\in \dom T) \quad \|Tx-Ty\|^2 + \|(\Id-T)x-(\Id-T)y\|^2 \leq \|x-y\|^2,
$$
where $\dom T\subseteq H$. 
The function $f:H\to\RX$ is said to be \emph{proper} if $\dom f:=\{x\in H
\ST f(x)<\infty\}\neq\varnothing$. The \emph{Fenchel conjugate} of the
function $f$ is the function $f^*$ defined by
\begin{equation}
f^*(u)=\sup_{x\in H}\big(\scal{u}{x}-f(x)\big).
\end{equation}
We set $q(\cdot)=\frac{1}{2}\|\cdot\|^2$. The {\em Moreau envelope} of $f$ is
the function defined by the {\em infimal convolution}
\begin{equation}
e_f(s)=(f\square q)(s)=\inf_{x\in H}\big(f(x)+q(s-x)\big).
\end{equation}
The {\em subdifferential} of the proper function $f$ is the mapping $\partial
f:H\rightrightarrows H$ defined by
$$
\partial f(x)=\big\{u\in H\ \big|\  f(x)+\scal{u}{y-x}\leq f(y),\ \ \forall y\in H\big\}.
$$
The \emph{indicator function} of a subset $C$ of $H$ is the function
$\iota_C:H\to\RX$ which vanishes on $C$ and equals $+\infty$ on
$H\smallsetminus C$.

\begin{fact}[Minty's Theorem {\rm\cite[Theorem~21.1]{BC2017}}]
 Let $A: H \rightrightarrows H$ be monotone. Then $A$ is maximally
 monotone if and only if $\ran (\Id + A) = H$.
\end{fact}

\begin{fact} {\rm (\cite[Proposition~23.8]{BC2017})}\label{f:firm_vs_mono}
	Let $A:H \rightrightarrows H$. Then 
	\begin{enumerate}
		\item\label{f:firm_vs_mono-i} $J_A$ is firmly nonexpansive if and only if $A$ is monotone;
		\item\label{f:firm_vs_mono-ii} $J_A$ is firmly nonexpansive and $\dom J_A=H$ if and only if $A$ is maximally monotone.
	\end{enumerate}
\end{fact}

Let $f$ be a proper lower semicontinuous convex function. 
The proximity operator \cite[Definition~12.23]{BC2017} of $f$ is defined by
\begin{equation}
\prox_f:H\to H:
x\mapsto \prox_f x=\argmin_{y\in H}\big(f(y)+q(y-x)\big).
\end{equation}
For all $s\in H$, \cite[Proposition~12.15]{BC2017} implies that there is a unique minimizer of $f(\cdot)+q(s-\cdot)$ over all $x\in H$; thus, the proximity operator of $f$ is well defined. Furthermore, we also have $\prox_f=J_{\partial f}$.

Additional properties of the Moreau envelope are:
\begin{fact}[Moreau envelope] 
Let $f$ be a proper lower semicontinuous convex function. The following
 assertions hold:
\begin{enumerate}
\item (Moreau decomposition) $e_{f}+e_{f^*}=q$.
\item $x=\prox_f s$\quad $\iff$\quad $e_{f}(s)=f(x)+q(s-x)$.
\item {\rm(\cite[Proposition~12.30]{BC2017})} $e_f$ is Fr\'echet differentiable with $\nabla e_f=\Id-\prox_f$.
\end{enumerate}
\end{fact}

Finally, we set the marginal projections $P_i\colon X\to X_i\colon (x_1,\ldots,x_N)\mapsto x_i$ for $i$ in $\{1,\ldots, N\}$ and the two-marginal projections $P_{i,j}\colon X\to X_i\times X_j\colon (x_1,\ldots,x_N)\mapsto
(x_i,x_j)$ for $i<j$ in $\{1,\ldots,N\}$. 
Given a subset $\Gamma$ of $X$, 
we set 
\begin{equation}\label{gammai}
\Gamma_i= P_i(\Gamma)
\qquad\text{and}\qquad
\Gamma_{i,j}=P_{i,j}(\Gamma)
\end{equation}
We also define $A_{i,j}:X_i\rightrightarrows X_j$ via
$$
\gra A_{i,j}=\Gamma_{i,j}.
$$
The notation $A_i$ is reserved for a different purpose and introduced 
in Section~\ref{s:multi-mar-Minty}.

\section{A characterization of multi-marginal $c$-monotonicity and Minty type criteria for  $c$-monotonicity}\label{s:multi-mar-Minty}

Let $\sumop :H\times H\to H$ be the mapping defined by $\sumop(x,y)=x+y$. For
any mapping $A:H\rightrightarrows H$, we have the identity
\cite[Lemma~12.14]{Rock-Wets}
\begin{equation}
J_{A^{-1}}=\Id-J_A.
\end{equation}
If, in addition, $A$ is monotone, then by Fact~\ref{f:firm_vs_mono}, $J_A$ and $J_{A^{-1}}$ are single-valued, thus,
\begin{equation}\label{e:JAJA-1}
J_{A}+J_{A^{-1}}=\Id|_{\sumop(\gra A)},
\end{equation}
which is equivalent to $\gra A$ being parameterized by
\begin{equation}\label{e:graA_para}
\gra A=\menge{(J_As,J_{A^{-1}}s)}{s\in \sumop(\gra A)}.
\end{equation} 

Given a set $\Gamma\subseteq X$, we now associate with $\Gamma$ monotone
mappings as follows. 
\begin{definition}
	Let $\Gamma\subseteq X$ be a set. For each index set $\varnothing \neq K \subsetneq I$, we define the mapping $A_K:H\rightrightarrows H$ by 
	\begin{equation}\label{e:monotone_condition}
	\gra A_K=\bigg\{\Big(\sum_{i\in K}x_i,\sum_{i\in I\setminus K}x_i\Big)\ \Big|\ (x_1,\ldots,x_N)\in \Gamma\ \bigg\}
	\end{equation}
	and for each $i\in I$ we set $A_i=A_{\{i\}}$.
\end{definition}

Our first aim is to characterize the $c$-monotonicity of a set $\Gamma$ in terms of the monotonicity of its $A_K$'s, and furthermore, extend \eqref{e:JAJA-1} and \eqref{e:graA_para} to the multi-marginal settings. To this end we will employ the sum mapping
\begin{equation}\label{e:sum_func}
\sumop:X\to H:(x_1,\ldots,x_N)\mapsto \sum_{i\in I}x_i,
\end{equation}
and the following fact which follows by a straightforward computation (see,
e.g.,~\cite[Fact 3.3]{BBW2}).

\begin{fact}\label{mono under shift}
	Let $x\in X$. If the subset $\Gamma$ of $X$ is $n$-$c$-cyclically monotone,
	then so is $\Gamma+x$.
\end{fact}

\begin{lemma}\label{l:c-mono_iff_mono}
	Let $\Gamma\subseteq X$ be a set. Then the following assertions are equivalent:
	\begin{enumerate}
		\item\label{l:c-mono_iff_mono-i} $\Gamma$ is $c$-monotone;
		\item\label{l:c-mono_iff_mono-ii} For each $\varnothing \neq K \subsetneq I$, the mapping $A_K$ is monotone;
		\item\label{l:c-mono_iff_mono-iii} For each $\varnothing \neq K \subsetneq I$, the mapping $J_{A_K}:\sumop(\Gamma)\to H$ is firmly nonexpansive.
	\end{enumerate}
	In this case,
	\begin{equation}\label{e:sumJAi=Id}
		J_{A_1}+\cdots + J_{A_N}=\Id|_{\sumop(\Gamma)},
	\end{equation}
	equivalently, $\Gamma$ can be parameterized by
	\begin{equation}\label{e:Gamma_para}
	\Gamma = \menge{(J_{A_1}s,\ldots,J_{A_N}s)}{s\in \sumop(\Gamma)};
	\end{equation}
	and, furthermore, for each $\varnothing \neq K \subsetneq I$, 
	\begin{equation}\label{e:JAK=sumJAi}
	J_{A_K}=\sum_{i\in K}J_{A_i}.
	\end{equation}
\end{lemma}

\begin{proof}
 \ref{l:c-mono_iff_mono-i} $\iff$ \ref{l:c-mono_iff_mono-ii}:
 First we characterize the $c$-monotone relations of the set $\{z,0\}$ in $X$. We employ a similar computation to the one in~\cite[Lemma 4.1]{BBW2}: For $z=(z_1,\ldots,z_N)\in X$ and $\varnothing \neq K \subsetneq I$ we set $z^K=(z^K_1,\ldots,z^K_N)\in X$ by 
	\begin{equation*}
	z^K_i=\begin{cases}
	z_i ,& i\in K;\\
	0 ,& i\in I\setminus K.
	\end{cases}
	\end{equation*}
	From Definition~\ref{monotonicity definitions} it follows that $\{z,0\}$ is $c$-monotone if and only if for each  $\varnothing\neq K\subsetneq I$
	\begin{align*}
	0&\leq c(z)+c(0)-c(z^K)-c(z^{I\setminus K})\\[1mm]
	&=\sum_{i,j\in I,\ i<j}\scal{z_i}{z_j}+0-\sum_{i,j\in I,\ i<j}\scal{z^K_i}{z^K_j}-\sum_{i,j\in I,\ i<j}\scal{z^{I\setminus K}_i}{z^{I\setminus K}_j}\\[1mm]
	&=\sum_{i,j\in I,\ i<j}\scal{z_i}{z_j}-\sum_{i,j\in K,\ i<j}\scal{z_i}{z_j}-\sum_{i,j\in I\setminus K,\ i<j}\scal{z_i}{z_j}=\bigg\langle\sum_{i\in K}z_i,\sum_{i\in I\setminus K}z_i\bigg\rangle.
	\end{align*}
	
	In general, from Definition~\ref{monotonicity definitions} it follows that the set $\Gamma\subseteq X$ is $c$-monotone if and only if for any $x\in\Gamma$ and $y\in\Gamma$, the set $\{x,y\}$ is $c$-monotone, which, in turn, by invoking Fact~\ref{mono under shift}, is equivalent to the set $\{x-y,0\}$ being $c$-monotone. Summing up, we see that $
	\Gamma$ is $c$-monotone if and only if for any $x=(x_1,\ldots,x_N),\ y=(y_1,\ldots,y_N)\in\Gamma$ and any $\varnothing \neq K \subsetneq I$, by letting $z=x-y$, 
	$$
	0\leq\bigg\langle\sum_{i\in K}x_i-\sum_{i\in K}y_i,\sum_{i\in I\setminus K}x_i-\sum_{i\in I\setminus K}y_i\bigg\rangle,
	$$
	i.e., $A_K$ is monotone.
	
 \ref{l:c-mono_iff_mono-ii} $\iff$ \ref{l:c-mono_iff_mono-iii}: By the
 definition of $A_K$, it follows that $\dom J_{A_K}=\sumop(\Gamma)$. Thus, the
 equivalence \ref{l:c-mono_iff_mono-ii} $\iff$ \ref{l:c-mono_iff_mono-iii}
 follows immediately from Fact~\ref{f:firm_vs_mono}\ref{f:firm_vs_mono-i}.
	
	Finally, \eqref{e:sumJAi=Id}, \eqref{e:Gamma_para} and \eqref{e:JAK=sumJAi} follow from \ref{l:c-mono_iff_mono-iii} and the definition of $A_K$.
\end{proof}

We now address maximal $c$-monotonicity. Equivalent statements of Minty's
characterization are: Let $A:H\rightrightarrows H$ be a monotone mapping.
Then $A$ is maximally monotone if and only if
$$
\sumop\big(\gra(A)\big)=H,
$$
equivalently,
$$
\gra(A)+\gra(-\Id)=H\times H.
$$ 

In order to extend our discussion of these formulas into the multi-marginal settings we will employ the following definitions and notations. We denote by $\Delta$ the subset of $X=X_1\times\cdots\times X_N$ defined by
\begin{equation}\label{e:Delta}
\Delta=\big\{(x,\ldots,x)\big|\ x\in H\big\}.\ \ \ \ \ \ \ \text{Consequently,}\ \ \ \ \ \Delta^{\perp}=\Big\{(x_1,\ldots,x_N)\in X\ \Big|\ \sum_{i=1}^N x_i=0\ \Big\}.
\end{equation}

\begin{corollary}\label{c:unique_mono_decomp}
Let $\Gamma\subseteq X$ be a $c$-monotone set. Then for every $u,v\in\Gamma$,
		\begin{equation}
		u-v\in \Delta^\perp\qquad \iff\qquad u=v.
		\end{equation}
\end{corollary}

\begin{proof}
Let $u=(u_1,\ldots,u_N)$ and $v=(v_1,\ldots,v_N)$ belong to $\Gamma$ and suppose that 
$$
u-v=d=(d_1,\ldots,d_N)\in\Delta^\perp.
$$
 We prove that $d_i=0$ for each $1\leq i\leq N$. To this end, set $1\leq
 i_0\leq N$. By Lemma~\ref{l:c-mono_iff_mono}, $A_{i_0}$ is monotone.
 Consequently we see that
	 \begin{equation*}
	 0\leq \bigg\langle u_{i_0}-v_{i_0},\sum_{i\neq i_0}u_i-\sum_{i\neq i_0}v_i\bigg\rangle =\bigg\langle d_{i_0},\sum_{i\neq i_0}d_i\bigg\rangle=\scal{d_{i_0}}{-d_{i_0}}=-\|d_{i_0}\|^2\leq 0.
	 \end{equation*}	 
\end{proof}	

Combining Lemma~\ref{l:c-mono_iff_mono} and
Corollary~\ref{c:unique_mono_decomp} with classical two-marginal monotone
operator theory, we arrive at the following result.

\begin{theorem}[multi-marginal maximal $c$-monotonicity]
\label{t:MMMM}
	Let $\Gamma\subseteq X$ be a $c$-monotone set. Then the following assertions are equivalent:
	\begin{enumerate}
		\item\label{t:MMMM-i} For each $\varnothing \neq K \subsetneq I$ the mapping $A_K$ defined by \eqref{e:monotone_condition} is maximally monotone;
		
		\item\label{t:MMMM-ii} There exists $\varnothing \neq K \subsetneq I$ such that the mapping $A_K$ is maximally monotone; 
		
		\item\label{t:MMMM-iii} $\Gamma+\Delta^\perp=X$;
		
		\item\label{t:MMMM-iv} $J_{A_{1}}+\cdots+J_{A_{N}}=\Id$;
			
		\item\label{t:MMMM-v} For each $\varnothing \neq K \subsetneq I$ the firmly nonexpansive mapping $J_{A_K}:H\to H$ has full domain and $J_{A_K}=\sum_{i\in K}J_{A_i}$;
			
		\item\label{t:MMMM-vi} $\sumop(\Gamma)=H$.
	\end{enumerate} 
	In this case, $\Gamma$ is maximally $c$-monotone.
\end{theorem}

\begin{proof}
\ref{t:MMMM-i} $\Rightarrow$ \ref{t:MMMM-ii} is trivial.
	
\ref{t:MMMM-ii} $\Rightarrow$ \ref{t:MMMM-iii}: Suppose that $A_K$ is maximally monotone and let $a=(a_1,\ldots,a_N)\in X$. We will prove that there exist $x=(x_1,\ldots,x_N)\in\Gamma$ and $d=(d_1,\ldots,d_N)\in\Delta^\perp$ such that $x+d=a$. Indeed, the maximal monotonicity of $A_K$ implies that $\ran(A_K+\Id)=H$. Consequently, by the definition of $A_K$, there exists $x=(x_1,\ldots,x_N)\in\Gamma$ such that $\sum_{i=1}^N x_i=\sum_{i=1}^N a_i $. For each $1\leq i\leq N$ we let $d_i=a_i-x_i$. Then $\sum_{i=1}^N d_i=0$, that is $d=(d_1,\ldots,d_N)\in\Delta^\perp$ and $x+d=a$. 

\ref{t:MMMM-iii} $\Rightarrow$ \ref{t:MMMM-iv}: Fix $1\leq i_0\leq N$. We prove that $A_{i_0}+\Id$ is onto. Indeed, let $s\in H$. We prove that there exists $x=(x_1,\ldots,x_N)\in\Gamma$ such that $(x_{i_0},s)\in\gra(A_{i_0}+\Id)$. Indeed, let $h=(h_1,\ldots,h_N)\in X$ such that $\sum_{i=1}^N h_i=s$. Then \ref{t:MMMM-iii} implies the existence of $x=(x_1,\ldots,x_N)\in\Gamma$ and $d=(d_1,\ldots,d_N)\in\Delta^\perp$ such that $x+d=h$. Consequently, $\sum_{i=1}^N x_i=s$ which implies that $(x_{i_0},s)=\Big(x_{i_0},\sum_{i=1}^N x_i\Big)\in\gra(A_{i_0}+\Id)$. Thus, since $A_{i_0}$ is monotone, we conclude that its resolvent $J_{A_{i_0}}$ is firmly nonexpansive and has full domain. This is true for each $1\leq i_0\leq N$ and since for any $s\in H$ there exists $x=(x_1,\ldots,x_N)\in\Gamma$ such that $\sum_{i=1}^N x_i=s$, we conclude that $\sum_{i=1}^N J_{A_i}(s)=\sum_{i=1}^N x_i=s$, that is, \ref{t:MMMM-iv} holds.

\ref{t:MMMM-iv} $\Rightarrow$ \ref{t:MMMM-v}: 
Since $A_K$ is monotone for every $\varnothing \neq K \subsetneq I$, the resolvent $J_{A_K}$ is firmly nonexpansive and \ref{t:MMMM-iv} implies it has full domain. Furthermore, by employing our notations from the previous step, we see that for every $s\in H$, $\sum_{i\in K} J_{A_i}(s)=\sum_{i\in K} x_i=J_{A_K}(s)$, that is, we have arrived at \ref{t:MMMM-v}. 

\ref{t:MMMM-v} $\Rightarrow$ \ref{t:MMMM-i}: Let $\varnothing \neq K \subsetneq I$. Since  the resolvent $J_{A_K}$ is firmly nonexpansive and has full domain, $A_K$ is maximally monotone. 

Summing up, we have established \ref{t:MMMM-i} $\Rightarrow$ \ref{t:MMMM-ii} $\Rightarrow$ \ref{t:MMMM-iii} $\Rightarrow$ \ref{t:MMMM-iv} $\Rightarrow$ \ref{t:MMMM-v} $\Rightarrow$ \ref{t:MMMM-i}. 

\ref{t:MMMM-iv} $\Rightarrow$ \ref{t:MMMM-vi}:
Since for each $1\leq i\leq N$,  $\dom(J_{A_i})=S(\Gamma)$, then \ref{t:MMMM-iv} $\Rightarrow$ \ref{t:MMMM-vi}. 

\ref{t:MMMM-vi} $\Rightarrow$ \ref{t:MMMM-iii}:
Suppose that $\sumop(\Gamma)=H$ and let $y\in X$. Then there exist $x\in\Gamma$
such that $\sumop(y)=\sumop(x)$. Consequently, $y-x\in\Delta^\perp$, which implies that
$y=x+(y-x)\in\Gamma+\Delta^\perp$.

Finally, we prove that \ref{t:MMMM-iii} implies the maximal $c$-monotonicity
of $\Gamma$. Indeed, suppose that $u$ is $c$-monotonically related to
$\Gamma$. We then write $u=d+v$ where $d\in\Delta^\perp$ and
$v\in\Gamma$. Since $u,v\in \Gamma\cup\{u\}$ which is $c$-monotone and
$u-v\in\Delta^\perp$, Corollary~\ref{c:unique_mono_decomp} implies that
$u=v\in\Gamma$.
\end{proof}

\begin{remark}
To the best of our knowledge, the question whether the multi-marginal
generalization of the other direction of Minty's characterization of maximal
monotonicity holds, namely, whether the maximal $c$-monotonicity of the set
$\Gamma$ implies that $\Gamma+\Delta^\perp=X$, equivalently, that $J_{A_{1}}+\cdots+J_{A_{N}}=\Id$, is still open.
\end{remark}

\begin{remark}
In the partition of the identity in~\eqref{e:sumJAi=Id} and in
Theorem~\ref{t:MMMM}\ref{t:MMMM-iv} we conclude from~\eqref{e:JAK=sumJAi} and
Theorem~\ref{t:MMMM}\ref{t:MMMM-v} that any partial sum of the firmly
nonexpansive mappings is also firmly nonexpansive. This is not the case for
general partitions of the identity into sums of firmly nonexpansive mappings;
indeed,
an example where partial sums of a partition of the identity into firmly
nonexpansive mappings fail to be firmly nonexpansive is provided
in~\cite[Example 4.4]{BBW1}. We elaborate further on this in
Example~\ref{e:failed partial sum} below.
\end{remark}

\section{Multi-marginal maximal $c$-monotonicity via continuity}
\label{s:MMMM_via_cont}

In the classical two-marginal case an important class of maximally monotone
operators is the one of continuous monotone operators. A continuity
criterion guarantees maximality in the multi-marginal framework as well:

\begin{theorem}\label{continuous monotone is maximal}
	Let $\Gamma\subseteq X$ be a $c$-monotone set. Suppose that $\Gamma$ is the graph of a continuous mapping $T=(T_2,\ldots,T_N):X_1\to\Pi_{i=2}^N X_i$, i.e.,
	\begin{equation*}
	\Gamma=\gra(T)=\big\{(x,T_2x,\ldots,T_N x)\big|\ x\in H\big\}
	\end{equation*}
	where for each $2\leq i\leq N$ the mapping $T_i:H\to H$ is continuous. Then $\Gamma$ is maximally $c$-monotone.
\end{theorem}

We provide two proofs for Theorem~\ref{continuous monotone is maximal}. We begin with a direct proof.

\begin{proof}
 Let $u=(u_1,\ldots,u_N)$ be $c$-monotonically related to $\Gamma$. We prove
 that $u\in\Gamma$. Since $A_1$, induced from the $c$-monotone set $\Gamma\cup\{u\}$, is monotone,
	\begin{equation*}
	\forall x\in H,\quad
	0\leq\bigg\langle u_1-x,\ \sum_{i=2}^N (u_i-T_i x)\bigg\rangle.
	\end{equation*}
	For $t>0$ we let $x_t=u_1+t\sum_{i=2}^N(u_i-T_i u_1)$. 
	Then $x_t\longrightarrow u_1$ as $t\to0^+$ and
	\begin{align*}
	0\leq t\bigg\langle \sum_{i=2}^N(T_i u_1-u_i),\ \sum_{i=2}^N (u_i-T_i x_t)\bigg\rangle.
	\end{align*}
	Since each $T_i$ is continuous, we deduce that
	\begin{align*}
	0&\leq \bigg\langle \sum_{i=2}^N(T_i u_1-u_i),\ \sum_{i=2}^N (u_i-T_i x_t)\bigg\rangle\\[2mm]
	&\xrightarrow[]{t\to 0^+}\ \bigg\langle \sum_{i=2}^N(T_i u_1-u_i),\ \sum_{i=2}^N (u_i-T_i u_1)\bigg\rangle=-\bigg\|\sum_{i=2}^N (u_i-T_i u_1)\bigg\|^2,
	\end{align*}
	which implies 
	\begin{equation*}
	\sum_{i=2}^N u_i=\sum_{i=2}^N T_i u_1;
	\end{equation*}
	equivalently,
	\begin{equation}
	(u_1,\ldots,u_N)-(u_1,T_2u_1,\ldots,T_Nu_1)\in\Delta^\perp.
	\end{equation}
	Thus, by Corollary~\ref{c:unique_mono_decomp}, we have
	$
	(u_1,\ldots,u_N)=(u_1,T_2u_1,\ldots,T_Nu_1)\in\gr T.
	$
\end{proof}

The second proof of Theorem~~\ref{continuous monotone is maximal} employs the classical two-marginal fact that a monotone and continuous mapping is maximally monotone \cite[Corollary~20.28]{BC2017}, Lemma~\ref{l:c-mono_iff_mono} and Theorem~\ref{t:MMMM}.

\begin{proof}
	Since $A_1(x)=\sum_{i=2}^N T_i(x)$ for every $x\in H$, by employing Lemma~\ref{l:c-mono_iff_mono} it follows that $A_1$ is a monotone and continuous mapping, hence, maximally monotone. Consequently, by employing Theorem~\ref{t:MMMM} we conclude that $\Gamma$ is maximally monotone.
\end{proof}

\section{Maximal $c$-monotonicity of $c$-splitting sets, $c$-conjugate tuples and multi-marginal convex analysis}
\label{s:c-split_tuples}

 We begin our discussion of $c$-splitting tuples by a known observation regarding the subdifferentials of the splitting functions: As in~\cite{GS,KS,RU} we observe that if $(f_1,\ldots,f_N)$ is a $c$-splitting tuple of $\Gamma\subseteq X$, then given $x=(x_1,\ldots,x_N)\in\Gamma$ and for any $x_1'\in X_1$,
\begin{align*}
\sum_{i=1}^N f_i(x_i)&=c(x_1,\ldots,x_N)\\[1mm]
\text{and}
\qquad
c(x_1',x_2,\ldots,x_N)&\leq f_1(x_1')+\sum_{i=2}^N f_i(x_i).
\end{align*}
Summing up these two inequalities followed by simplifying, we see that
\begin{equation*}
f_1(x_1)+\scal{x_1'}{x_2+\cdots+x_N}\leq f_1(x_1')+\scal{x_1}{x_2+\cdots+x_N},\ \ \ \ \text{that is,}\ \ \ \ \ \ \sum_{i=2}^N x_i\in\partial f_1(x_1).
\end{equation*}
Similarly, we conclude that for each $1\leq i_0\leq N$, 
\begin{equation}\label{e:split_subdiff i}
\sum_{i\neq i_0} x_i\in\partial f_{i_0}(x_{i_0}).
\end{equation}
Since $\gra A_{i_0}=\Big\{\big( x_{i_0},\sum_{i\neq i_0}x_i\big)\ \Big|\ (x_1,\ldots,x_N)\in \Gamma\ \Big\}$, this implies
\begin{equation}\label{e:split_subdiff ii}	
	\gra(A_{i_0})\subseteq\gra(\partial f_{i_0}).
\end{equation}
Similar observations and $c$-monotonicity properties of $\Gamma$ from Section 2 are also related to the \emph{Wasserstein barycenter} as can be seen, for example, in~\cite{AC}. 

We continue our discussion by a characterization of $c$-splitting tuples and their generated $c$-splitting sets in terms of the Moreau envelopes of the splitting functions.

\begin{theorem}\label{envelopes inequality and equality}
	For each $1\leq i\leq N$, 
	let $f_i:X_i\to\RX$ be proper, lower semicontinuous, and convex.  
Then $c\leq\bigoplus_{i=1}^N f_i$ if and only if
\begin{equation}\label{e:env_ineq}
\forall s\in H,\quad
e_{f_1^*}(s)+\cdots+e_{f_N^*}(s)\leq q(s).
\end{equation}
Now assume this is the case, and  let $\Gamma\subseteq X$ be the $c$-splitting set generated  by $(f_1,\ldots,f_N)$. Then equality in~\eqref{e:env_ineq} holds if and only if $s=x_1+\cdots+x_N$ where $(x_1\ldots,x_N)\in\Gamma$. 
\end{theorem}

\begin{proof}
The inequality $c\leq\bigoplus_{i=1}^N f_i$ holds if and only if for all $(x_1,\ldots,x_N)\in X$,
\begin{align}
&c(x_1,\ldots,x_N)\leq \sum_{i=1}^N f_i(x_i)\label{e:c_leq_fi}\\
\iff\quad
&q(x_1+\cdots+x_N)=c(x_1,\ldots,x_N)+\sum_{i=1}^{N}q(x_i)\leq \sum_{i=1}^{N}\big(f_i(x_i)+q(x_i)\big).\label{e:q_leq_fi+q}
\end{align}
We see that \eqref{e:c_leq_fi} holds with equality only when $(x_1,\ldots,x_N)\in\Gamma$ if and only if \eqref{e:q_leq_fi+q} holds with equality only when $(x_1,\ldots,x_N)\in\Gamma$. Let $\phi:X\to\RR$ be defined by 
$$
\phi(x_1,\ldots,x_N)=q(x_1+\cdots+x_N).
$$
Then, using \cite[Corollary~15.28(i)]{BC2017}, we have 
\begin{equation*}
\forall(x_1,\ldots,x_N)\in X,\quad
\phi^*(x_1,\ldots,x_N)=q(x_1)+\iota_{\Delta}(x_1,\ldots,x_N).
\end{equation*}
Since for each $1\leq i\leq N$, $(f_i+q)^*=e_{f_i^*}$ (see, for example, \cite[Proposition~14.1]{BC2017}), we arrive at
\begin{equation*}
\Big(\bigoplus_{i=1}^N (f_i+q)\Big)^*=\bigoplus_{i=1}^N (f_i+q)^*=\bigoplus_{i=1}^N e_{f_i^*}.
\end{equation*}
Consequently, (classical) Fenchel conjugation transforms~\eqref{e:q_leq_fi+q} into~\eqref{e:env_ineq} and vise versa. 

We now address the case of equality in~\eqref{e:env_ineq}. Let $(x_1\ldots,x_N)\in X$ and $s=x_1+\cdots+x_n$. Then for each $1\leq i\leq N$, by the Fenchel-Young inequality,
\begin{equation}\label{Y-F for envelopes}
\scal{s}{x_i}\leq (f_i+q)^*(s)+(f_i+q)(x_i)=e_{f_i^*}(s)+(f_i+q)(x_i)
\end{equation}
with equality if and only if $x_i\in\partial (f_i+q)^*(s)$, i.e., since $(f_i+q)^*=e_{f_i^*}$ is Fr\'echet differentiable (see, e.g., \cite[Proposition~12.30]{BC2017}), $x_i=\nabla e_{f_i^*}(s)$ . 
By summing up~\eqref{Y-F for envelopes} over $i$, we obtain
\begin{equation}\label{multi-marginal Y-F for envelopes}
\scal{s}{s}=\sum_{i=1}^N\scal{s}{x_i}\leq\sum_{i=1}^N \big(e_{f^*_i}(s)+(f_i+q)(x_i)\big)
\end{equation}
with equality if and only if $x_i=\nabla e_{f_i^*}(s)$ for every $1\leq i\leq N$. 

($\Leftarrow$): Suppose that $x=(x_1,\ldots,x_N)$ is in the $c$-splitting set
$\Gamma$ generated by $(f_1,\ldots,f_N)$ and set $s=\sumop(x)$. We prove equality
in \eqref{e:env_ineq}. It follows from~\eqref{e:split_subdiff i} that for each
$1\leq i\leq N$,
\begin{equation}
\sum_{j\neq i}x_j=s-x_i\in\partial f_i(x_i)
\quad\iff\quad
s\in\partial (f_i+q)(x_i),
\end{equation}
which, in turn, implies that $x_i\in\partial (f_i+q)^*(s)$, that is, $x_i=\nabla e_{f_i^*}(s)$. Since in this case there is equality in~\eqref{multi-marginal Y-F for envelopes} and in~\eqref{e:q_leq_fi+q}, we obtain equality in~\eqref{e:env_ineq}. 

($\Rightarrow$): Let $s\in H$ be a point where equality in~\eqref{e:env_ineq}
holds. Since $\sum_{i=1}^N e_{f_i^*}$ and $q$ are Fr\'echet differentiable
and $\sum_{i=1}^N e_{f_i^*}\leq q$, then at the point of equality $s$ we have
$$
\nabla\Big(\sum_{i=1}^N e_{f_i^*}\Big)(s)=\nabla q(s)=s.
$$ 
For each $1\leq i\leq N$, set $x_i=\prox_{f_i}(s)=\nabla e_{f^*_i}(s)$ (see, e.g., \cite[eq~(14.7)]{BC2017}). Then it follows that $s=x_1+\cdots+x_N$. Thus, in order to complete the proof it is enough to prove that $(x_1,\ldots,x_N)\in\Gamma$ or, equivalently, that there is equality in~\eqref{e:q_leq_fi+q}. Indeed, Moreau's decomposition (see, e.g., \cite[Remark~14.4]{BC2017}) implies that $e_{f_i}+e_{f_i^*}=q$ for each $1\leq i\leq N$. Consequently,
\begin{equation*}
\sum_{i=1}^N e_{f_i^*}(s)= q(s)
\qquad
\text{is equivalent to}
\qquad
\sum_{i=1}^N e_{f_i}(s)=(N-1)q(s).
\end{equation*}
We also note that for each $1\leq i\leq N$, $x_i=\prox_{f_i}(s)$ implies that
\begin{equation*}
e_{f_i}(s)=\min_{x\in H}\big(f_i(x)+q(s-x)\big)=f_i(x_i)+q(s-x_i).
\end{equation*} 
Thus, we arrive at
\begin{align*}
&\sum_{i=1}^N \big(q(s-x_i)+f_i(x_i)\big)=(N-1)q(s)\\[1mm]
\Leftrightarrow\qquad &-\sum_{i=1}^N\scal{s}{x_i}+\sum_{i=1}^N\big(f_i(x_i)+q(x_i)\big)=-q(s)\\[1mm]
\Leftrightarrow\qquad &\sum_{i=1}^N\big(f_i(x_i)+q(x_i)\big)=q(s).
\end{align*}
\end{proof}

We now address $c$-conjugation.

\begin{definition}[$c$-conjugate tuple]
	For each $1\leq i\leq N$, let $f_i:X_i\to\RX$ be a proper function. We say that $(f_1,\ldots,f_N)$ is a {\em $c$-conjugate tuple} if for each $1\leq i_0\leq N$,
	\begin{equation*}
	f_{i_0}(x_{i_0})=\Big(\bigoplus_{i\neq i_0}f_i\Big)^c(x_{i_0}) :=\sup_{i\neq i_0,\ x_i\in X_i}\ c(x_1,\ldots,x_{i_0},\ldots,x_N)-\sum_{i\neq i_0}f_i(x_i),\ \ \ \ \ \ x_{i_0}\in X_{i_0}.
	\end{equation*} 
\end{definition}
It follows that if $(f_1,\ldots,f_N)$ is a $c$-conjugate tuple, then $f_i$ is lower semicontinuous and convex for each $1\leq i\leq N$. Furthermore, it is known (see \cite{GS} and \cite{CN}) that given a $c$-splitting tuple $(u_1,\ldots,u_N)$ of a set $\Gamma\subseteq X$, it can be relaxed into a $c$-conjugate $c$-splitting tuple $(f_1,\ldots,f_N)$ of $\Gamma$ by setting
\begin{equation*}
f_{1}=\Big(\bigoplus_{2\leq i\leq N}u_i\Big)^c
\end{equation*}
inductively,
\begin{equation*}
f_{i_0}=\Big(\bigoplus_{1\leq i\leq i_0-1}f_i \oplus \bigoplus_{i_0+1\leq i\leq N}u_i\Big)^c
\qquad\text{for}\ \ 2\leq i_0\leq N-1,
\end{equation*}
and finally
\begin{equation*}
f_N=\Big(\bigoplus_{1\leq i\leq N-1}f_i\Big)^c.
\end{equation*}

In the case $N=2$, let $f_1:X_1\to\RX$ be proper, lower semicontinuous and convex, let $f_2=f_1^*:X_2\to\RX$ be its conjugate and let $\Gamma=\gra(\partial f_1)\subseteq X_1\times X_2$. Then it is well known that $\Gamma$ is maximally monotone, see, e.g., \cite[Theorem~20.25]{BC2017}. Since $f_1=f_1^{**}=f_2^c$ and also $f_2=f_1^c$, then we can restate as follows: 
\begin{quote}
	Let $\Gamma\subseteq X_1\times X_2$ be the $c$-splitting set generated by the $c$-conjugate pair $(f_1,f_2)$. Then $\Gamma$ is maximally $c$-monotone and determines its $c$-conjugate $c$-splitting tuple $(f_1,f_2)$ uniquely up to an additive constant pair $(\rho,-\rho)$ with $\rho\in\RR$.
\end{quote}
A generalization to an arbitrary $N\geq 2$ would be
\begin{quote}
	Let $\Gamma\subseteq X$ be the $c$-splitting set generated by the $c$-conjugate tuple $(f_1,\ldots,f_N)$. Then $\Gamma$ is maximally $c$-monotone and determines its $c$-conjugate $c$-splitting tuple $(f_1,\ldots,f_N)$ uniquely up to an additive constant tuple $(\rho_1,\ldots,\rho_N)$ such that $\sum_{i=1}^N \rho_i=0$.  
\end{quote} 
To the best of our knowledge, whether or not this latter assertion is true in
general is still open. We do, however, provide a positive answer in a more
particular case in Theorem~\ref{3-marginal smooth conjugate} and additional
insight in Theorem~\ref{t:split_max_c-mono}. 

Furthermore, we note that in the case $N=2$,
given a conjugate pair $(f_1,f_2)$, Moreau's decomposition can be restated as
\begin{equation}\label{e:moreau_decomp}
e_{f_1^*}+e_{f_2^*}=q
\qquad\text{and}\qquad \prox_{f_1}+\prox_{f_2}=\Id.
\end{equation}
Combining our discussion with Theorems~\ref{envelopes inequality and equality} and \ref{l:c-mono_iff_mono}, we arrive at the following generalized multi-marginal convex analytic assertions which, in particular, generalize the decomposition \eqref{e:moreau_decomp}. To this end, we again recall that for each $1\leq i_0\leq N$, 
$$\gra A_{i_0}=\Big\{\big( x_{i_0},\sum_{i\neq i_0}x_i\big)\ \Big|\ (x_1,\ldots,x_N)\in \Gamma\ \Big\}.$$

\begin{theorem}\label{t:split_max_c-mono}
	For each $1\leq i\leq N$, let $f_i:X_i\to\RX$ be 
	convex, lower semicontinuous, and proper. Suppose that $\Gamma\subseteq X$ is the $c$-splitting set generated  by $(f_1,\ldots,f_N)$. Then the following assertions are equivalent:
	\begin{enumerate}
		\item\label{t:split_max_c-mono-i} There exist $1\leq i_0\leq N$ such that $A_{i_0}$ is maximally monotone;
		
		\item\label{t:split_max_c-mono-ii} There exist $1\leq i_0\leq N$ such that $A_{i_0}=\partial f_{i_0}$;
		
		\item\label{t:split_max_c-mono-iii} $A_{i}=\partial f_{i}\ $ for each $1\leq i\leq N$;
		
		\item\label{t:split_max_c-mono-iv} $\prox_{f_1}+\cdots+\prox_{f_N}=\Id$;
		
		\item\label{t:split_max_c-mono-v} $e_{f_1^*}+\cdots+e_{f_N^*}=q$.
	\end{enumerate}	
	In this case
	\begin{enumerate}[label={\rm(\Alph*)}]
		\item\label{t:split_max_c-mono-A} $\Gamma$ is maximally $c$-monotone (and, consequently, maximally $c$-cyclically monotone);
		
  \item\label{t:split_max_c-mono-B} $(f_1,\ldots,f_N)$ is a $c$-conjugate
  $c$-splitting tuple of $\Gamma$. Moreover, $\Gamma$ determines its
  $c$-conjugate $c$-splitting tuple $(f_1,\ldots,f_N)$ uniquely up to an
  additive constant tuple $(\rho_1,\ldots,\rho_N)$ such that $\sum_{i=1}^N
  \rho_i=0$.
	\end{enumerate}
\end{theorem}

\begin{proof}
	\ref{t:split_max_c-mono-i} $\Rightarrow$ \ref{t:split_max_c-mono-ii}: $\partial f_{i_0}$ is monotone and $\gra(A_{i_0})\subseteq\gra(\partial f_{i_0})$ (see~\eqref{e:split_subdiff ii}). Consequently, since $A_{i_0}$ is maximally monotone,  it follows that $A_{i_0}=\partial f_{i_0}$. 
	
 \ref{t:split_max_c-mono-ii} $\Rightarrow$ \ref{t:split_max_c-mono-iii}:
 $A_{i_0}=\partial f_{i_0}$ is maximally monotone as the subdifferential of a
proper lower semicontinuous convex function. Consequently, it follows from
 Theorem~\ref{t:MMMM}\ref{t:MMMM-i}\&\ref{t:MMMM-ii} that $A_i$ is maximally
 monotone for each $1\leq i\leq N$. Now, $\partial f_i$ is monotone and
 $\gra(A_{i})\subseteq\gra(\partial f_{i})$ (see~\eqref{e:split_subdiff ii}).
 Consequently, since $A_i$ is maximally monotone, it follows that
 $A_i=\partial f_i$.
	
 \ref{t:split_max_c-mono-iii} $\Rightarrow$ \ref{t:split_max_c-mono-iv}:
 Follows from Theorem~\ref{t:MMMM}\ref{t:MMMM-i}\&\ref{t:MMMM-iv} since
 $A_i=\partial f_i$ is maximally monotone and $\prox f_i=J_{\partial
 f_i}=J_{A_i}$.
	
 \ref{t:split_max_c-mono-iv} $\Rightarrow$ \ref{t:split_max_c-mono-v}: By
 integrating \ref{t:split_max_c-mono-iv} we obtain the equality in
 \ref{t:split_max_c-mono-v} up to an additive constant.
 Theorem~\ref{envelopes inequality and equality} implies that equality in
 \ref{t:split_max_c-mono-v} holds on $S(\Gamma)$; thus, the additive constant
 vanishes.
	
 \ref{t:split_max_c-mono-v} $\Rightarrow$ \ref{t:split_max_c-mono-i}: By
 Theorem~\ref{envelopes inequality and equality} equality in
 \ref{t:split_max_c-mono-v} holds only on $\sumop(\Gamma)$. Consequently,
 \ref{t:split_max_c-mono-v} implies that $\sumop(\Gamma)=H$. By employing
 Theorem~\ref{t:MMMM}\ref{t:MMMM-vi}\&\ref{t:MMMM-i}, we obtain
 \ref{t:split_max_c-mono-i}.
	
	In this case Theorem~\ref{t:MMMM} also implies $\Gamma$ is maximally $c$-monotone. Thus, it remains to prove \ref{t:split_max_c-mono-B}. By our preliminary discussion there exists a $c$-conjugate $c$-splitting tuple $(h_1,\ldots,h_N)$ of $\Gamma$. From \ref{t:split_max_c-mono-iii} and from \eqref{e:split_subdiff ii} we conclude that $\gra(\partial f_{i})=\gra(A_{i})\subseteq\gra(\partial h_{i}) $ which, by maximality, implies that $\partial f_i=\partial h_i$ for each $1\leq i\leq N$. Here there exists a constant tuple $(\rho_1,\ldots,\rho_N)\in\RR^N$ such that $(f_1,\ldots,f_N)=(h_1,\ldots,h_N)+(\rho_1,\ldots,\rho_N)$. For $(x_1,\ldots,x_N)\in\Gamma$ the equality $\sum_{i=1}^N f_i(x_i)=\sum_{i=1}^N h_i(x_i)$ implies that $\sum_{i=1}^N \rho_i=0$. Consequently, the fact that for each $1\leq i_0\leq N$
	\begin{equation*}
	f_{i_0}-\rho_{i_0}=\Big(\bigoplus_{i\neq i_0}^N(f_{i}-\rho_{i})\Big)^c
	\end{equation*}
	implies that $(f_1,\ldots,f_N)$ is a $c$-conjugate tuple.
\end{proof}

We now provide a smoothness criteria in the 3-marginal case where Theorem~\ref{t:split_max_c-mono}\ref{t:split_max_c-mono-i}--\ref{t:split_max_c-mono-v}\&\ref{t:split_max_c-mono-B} are equivalent and imply maximal $c$-monotonicity. To this end we will employ the following facts.

\begin{fact}
{\rm (\cite[Theorem~14.19]{BC2017})}
\label{Toland}
	Let $g:H\to\RX$ be proper, let $h:H\to\RX$ be proper, lower semicontinuous and convex. Set
	\begin{equation*}
	f:H\to[-\infty,+\infty]:x\mapsto\begin{cases}
	g(x)-h(x), & \ \ \ x\in\dom(g);\\
	+\infty, & \ \ \ x\notin\dom(g).
	\end{cases}
	\end{equation*}
	Then
	\begin{equation*}
	f^*(y)=\sup_{v\in\dom(h^*)}\big(g^*(y+v)-h^*(v)\big).
	\end{equation*}
\end{fact}

\begin{fact}\label{Solov}
    {\rm(\cite[Corollary 2.3]{Sol})}
	Let $f:\RR^n\to\RR$ be proper and lower semicontinuous. If $f^*$ is essentially smooth, then $f$ is convex.
	{\color{blue} }
\end{fact}

\begin{theorem}\label{3-marginal smooth conjugate}
	Let $n\in\NN,\ N=3$ and $H=\RR^n$. Let $g:X_2\to\RX$ and $h:X_3\to\RX$ be proper, lower semicontinuous and convex functions. Suppose that $f=(g\oplus h)^c$ (in particular if $(f,g,h)$ is a $c$-conjugate triple) and that $f$ is essentially smooth. Let $\Gamma$ be the $c$-splitting set generated by $(f,g,h)$. Then assertions~\ref{t:split_max_c-mono-i}--\ref{t:split_max_c-mono-v} of Theorem~\ref{t:split_max_c-mono} hold and $\Gamma$ is maximally $c$-monotone.
\end{theorem}

\begin{proof}
	Since $f=(g\oplus h)^c$ and $\dom(g+q)^*=\dom(e_{g^*})=\RR^n$, then by employing Fact~\ref{Toland} in~\eqref{f(x+y) conjugates} and then Moreau's decomposition in~\eqref{Moreau's decomp for 3 conjugates} we see that
	\begin{align}
	(f+q)(x)&=
	\sup_{y,z\in\RR^n} \big(c(x,y,z)-g(y)-h(z)+q(x)\big)\nonumber\\
	&=\sup_{y,z\in\RR^n} \big(\scal{x}{y}+\scal{y}{z}+\scal{z}{x}+q(x)-g(y)-h(z)\big)\nonumber\\
	&=\sup_{y\in\RR^n}\big( \scal{x}{y}+h^*(x+y)+q(x)-g(y)\big)\nonumber\\
	&= \sup_{y\in\RR^n}\big( h^*(x+y)+q(x+y)-(g(y)+q(y))\big)\nonumber\\
	&=\big((h^*+q)^*-(g+q)^*\big)^*(x)\label{f(x+y) conjugates}\\
	&=(e_h-e_{g^*})^*(x)=(q-e_{g^*}-e_{h^*})^*(x).\label{Moreau's decomp for 3 conjugates}
	\end{align}
	Since $f+q$ is essentially smooth, Fact~\ref{Solov} implies that $q-e_{g^*}-e_{h^*}$ is convex. Consequently,
	$$
	e_{f^*}=(f+q)^*=(q-e_{g^*}-e_{h^*})^{**}=q-e_{g^*}-e_{h^*},
	$$
	that is, $e_{f^*}+e_{g^*}+e_{h^*}=q$.
\end{proof}

\begin{remark}
In our discussion in the last paragraph of Section~\ref{s:multi-mar-Minty} we pointed out that in the partition of the identity in~Theorem~\ref{t:MMMM}\ref{t:MMMM-iv} any partial sum of the firmly nonexpansive mappings is again firmly nonexpansive and, furthermore, that general partitions of the identity into firmly nonexpansive mappings partial sums may fail to be firmly nonexpansive. Thus, in the context of $c$-splitting sets a natural question is: Given a partition of the identity into proximal mappings, are partial sums also proximal mappings? Unlike general firmly nonexpansive mappings, a positive answer to this question is provided by~\cite[Theorem~4.2]{BBW1}. 
\end{remark}

\section{Examples, observations and remarks}
\label{s:ex}

We now apply our results in order to determine maximality of $c$-monotone
sets. Given a multi-marginal $c$-cyclically monotone set $\Gamma\subseteq X$,
the problem of constructing a $c$-splitting tuple is, in general, nontrivial.
Nevertheless, constructions which are independent of maximality and uniqueness considerations are available for some classes of
$c$-cyclically monotone sets (for example, see \cite{BBW2} for the case $N \geq 3$). We also note that $c$-splitting tuples can be constructed via~\eqref{e:split_subdiff ii} if it is known, in addition, that the antiderivatives $f_i$ are unique up to additive constants, as guaranteed by Theorem~\ref{t:split_max_c-mono}. Now, suppose that a $c$-splitting tuple is already given. The computation and
classification of the $c$-splitting tuple as being a $c$-conjugate tuple
were, thus far, nontrivial. We employ our results for such classifications in
the following examples. For these cases, we are able to conclude
$c$-conjugacy without additional nontrivial computations of multi-marginal
conjugates. In addition, we demonstrate finer aspects of multi-marginal
maximal monotonicity.

\begin{example}\label{ex:quadratics}
 For each $1\leq i\leq N$, set
	$X_i=\RR^d$ and let $Q_i\in\RR^{d\times d}$ be symmetric,
	positive definite, and pairwise commuting. 
	Set 
	\begin{equation*}
	\Gamma=\big\{(Q_1v,\ldots,Q_Nv)\ \big|\ v\in\RR^d\big\}.
	\end{equation*}
	For each $1\leq i\leq M$, define $M_i\in\RR^{d\times d}$ by 
	\begin{equation*}
	M_i=\Big(\sum_{k\neq i}Q_k\Big)Q_i^{-1}.
	\end{equation*}
	In~\cite[Example~3.4]{BBW2}, it was established that
	\begin{equation*}
	\forall (x_1,\ldots,x_N)\in X,\quad
	c(x_1,\ldots,x_N)=\sum_{1\leq i< j\leq N}\scal{x_i}{x_j}\leq\sum_{1\leq i\leq N}q_{M_i}(x_i),
	\end{equation*}
	where $q_{M_i}(x)=\frac{1}{2}\scal{x}{M_i x}$,
	and equality holds if and only if $(x_1,\ldots,x_N)\in\Gamma$.
	
	Thus, we conclude that $\Gamma$ is the $c$-splitting set generated by the tuple $(q_{M_1},\ldots,q_{M_N})$, and that $A_i=M_i=\nabla q_{M_i}$ for each $1\leq i\leq N$. Consequently, Theorem~\ref{t:split_max_c-mono} implies that $(q_{M_1},\ldots,q_{M_N})$ is a $c$-conjugate $c$-splitting tuple of $\Gamma$, and that $\Gamma$ is maximally $c$-monotone.
	
	The maximal $c$-monotonicity of $\Gamma$ is also implied by Theorem~\ref{continuous monotone is maximal} via continuity of a parametrization, say,
		\begin{equation*}
		\Gamma=\big\{(v,Q_2Q_1^{-1}v\ldots,Q_NQ_1^{-1}v)\ \big|\ v\in\RR^d\big\}.
		\end{equation*}
\end{example}

As a simple application of Example~\ref{ex:quadratics}, we now generalize the well-known classical fact that the only conjugate pair of the form $(f,f)$ is $(f,f)=(q,q)$ and that in this case the generated splitting set is the graph of the identity mapping.

\begin{corollary}[self $c$-conjugate tuple]
	The only $c$-conjugate tuple of the form $(f,\ldots,f)$ is 
	\begin{equation*}
	(f,\ldots,f)=(N-1)(q,\ldots,q).
	\end{equation*}
	In this case, the generated $c$-splitting set is $\Gamma=\Delta$.
\end{corollary}

\begin{proof}
 In the settings of Example~\ref{ex:quadratics} we let $Q_i=\Id$ for each
 $1\leq i\leq N$. Then $\Gamma=\Delta$ and $q_{M_i}=(N-1)q$ for each $1\leq
 i\leq N$. We conclude that $(N-1)(q,\ldots,q)$ is a $c$-conjugate
 $c$-splitting tuple and generates the $c$-splitting set $\Delta$. We now
 prove that it is the only $c$-conjugate tuple of this form. Let
 $(f,\ldots,f)$ be a $c$-conjugate tuple. Then for $1\leq i_0\leq N$ and for
 $x_{i_0}\in X_{i_0}$,
	\begin{equation}\label{self conjugates}
	f(x_{i_0})=\sup_{i\neq i_0,\ x_i\in H}\Big(c(x_1,\ldots,x_{i_0},\ldots,x_N)-\sum_{i\neq i_0}f(x_i)\Big).
	\end{equation}
	By letting $x_i=x_{i_0}$ for every $i$ in the supremum in~\eqref{self conjugates} we see that
	$$
	f(x_{i_0})\geq c(x_{i_0},\ldots,x_{i_0})-(N-1)f(x_{i_0})\ \ \ \ \ \ \ \Rightarrow\ \ \ \ \ \ \ Nf\geq N(N-1)q\ \ \ \ \ \Rightarrow\ \ \ \ \ \ f\geq (N-1)q.
	$$
	Consequently, 
	$$
	f=\Big(\bigoplus_{i\neq i_0} f\Big)^c\leq \Big(\bigoplus_{i\neq i_0} q\Big)^c=(N-1)q.
	$$
\end{proof}	

A similar type of construction to the one of Example~\ref{ex:quadratics},
however, a nonlinear one, is available when the marginals are one-dimensional.

\begin{example}
	For each $1\leq i\leq N$, let $\alpha_i:\RR\to\RR$ be a continuous, 
	strictly increasing and surjective function with $\alpha_i(0)=0$. Let $\Gamma$ be the curve in $\RR^N$ defined by
	$$
	\Gamma=\Big\{\big(\alpha_1(t),\ldots,\alpha_N(t)\big)\ \Big|\ t\in\RR\Big\}
	$$
	and for each $1\leq i\leq N$, let
	\begin{equation}\label{curves}
	f_i(x_i)=\int_0^{x_i} \bigg(\sum_{k\neq i} \alpha_k\big(\alpha^{-1}_i(t)\big)\bigg)dt.
	\end{equation}
In~\cite[Example 4.3]{BBW2}, it was established that
	\begin{equation}\label{e:curve_ineq}
	\sum_{1\leq i<j\leq N} x_ix_j \leq \sum_{i=1}^N\int_0^{x_i} \bigg(\sum_{k\neq i} \alpha_k\big(\alpha^{-1}_i(t)\big)\bigg)dt
	\ \ \ \ \ \ \ \ \ \ \ \ \ \forall (x_1,\ldots,x_N)\in\RR^N
	\end{equation}
and that equality in~\eqref{e:curve_ineq} holds if and only if $x_j=\alpha_j\big(\alpha_i^{-1}(x_i)\big)$ for every $1\leq i<j\leq N$, namely, if $(x_1,\ldots,x_N)\in\Gamma$. We now conclude that $\Gamma$ is the $c$-splitting set generated by the tuple $(f_1,\ldots,f_N)$ and that for each $1\leq i\leq N$,
	\begin{equation*}
	A_i=\nabla f_i=\sum_{k\neq i} \alpha_k\circ\alpha_i^{-1}.
	\end{equation*}
Consequently, Theorem~\ref{t:split_max_c-mono} implies that $(f_1,\ldots,f_n)$ is a $c$-conjugate $c$-splitting tuple of the maximally $c$-monotone curve $\Gamma$. Similar to Example~\ref{ex:quadratics}, the maximal $c$-monotonicity of $\Gamma$ can also be deduced via continuity.
\end{example}

A linear example of a different type, where none of the two marginal
projections of $\Gamma$ is monotone, but where, however, $\Gamma$ is $c$-cyclically
monotone, is available for $N=3$ and 2-dimensional marginals.

\begin{example}\label{ex:gamma ij not mono}
	Suppose that $N=3$ and that $X_1 = X_2=X_3=\RR^2$.  We set
	$$
	M_1=2\begin{pmatrix}
	1 & 0\\0 & 0
	\end{pmatrix},\ \ \ \
	M_2=2	\begin{pmatrix}
	1 & 0\\0 & 1
	\end{pmatrix},\ \ \ \ 
	M_3=\frac{1}{7}	\begin{pmatrix}
	8 & 3\\3 & 2
	\end{pmatrix}
	$$
	and
	$$
	\Delta_2=\big\{(a,a)\ \big|\ a\in\RR\big\} \subseteq\RR^2.
	$$
	Set 
	\begin{equation*}
	f_1=\iota_{\RR\times\{0\}}+q_{M_1},\qquad f_2=\iota_{\Delta_2}+q_{M_2}=\iota_{\Delta_2}+2q,\qquad
	\text{and}\qquad f_3=q_{M_3}.
	\end{equation*}
	
	\ \\
	Furthermore, set $v_1=\big((0,0),(-1,-1),(1,-5)\big)$, $v_2=\big((1,0),(2,2),(0,7)\big)$ and 
	$$
	\Gamma=\spa\{v_1,v_2\}=\Big\{\big((s,0),(2s-t,2s-t),(t,7s-5t)\big)\Big|\ s,t\in\RR\Big\}.
	$$
	It was established in~\cite[Example 3.5]{BBW2} that
	$$
	\scal{x_1}{x_2}+\scal{x_2}{x_3}+\scal{x_3}{x_1}\leq f_1(x_1)+f_2(x_2)+f_3(x_3)\ \ \ \ \ \ \ \ \ \text{for all} \ \ \ \ (x_1,x_2,x_3)\in\big(\RR^2\big)^3
	$$
	with equality if and only if $(x_1,x_2,x_3)\in\Gamma$, namely, $\Gamma$ is the $c$-splitting set generated by the tuple $(f_1,f_2,f_3)$ and that none of the two marginal projections $\Gamma_{1,2},\ \Gamma_{1,3}$ and $\Gamma_{2,3}$ of $\Gamma$, is monotone. 
	
We observe that the matrix representation of the mapping
	$$
	(t,7s-5t)\ \mapsto\ (s,0)+(2s-t,2s-t)\ \ \ \ \ \ \ s,t\in\RR
	$$ 
	is $M_3$. Consequently, we see that $A_3=M_3=\nabla f_3$. Thus, by employing Theorem~\ref{t:split_max_c-mono} we conclude that $(f_1,f_2,f_3)$ is a $c$-conjugate $c$-splitting tuple of the maximally $c$-monotone subspace $\Gamma$ of $\big(\RR^2\big)^3$.
\end{example}	

In all of our examples thus far, the set $\Gamma$ was a maximally $c$-monotone
$c$-splitting set. We now present maximally $c$-monotone sets which are not
$c$-splitting sets. To this end, we note the following simple fact: Suppose
that the set $\Gamma\subseteq X$ is $n$-$c$-monotone, then for each $1\leq
i_0\leq N$ the mapping $A_{i_0}:H\rightrightarrows H$ is $n$-monotone.
Indeed, let $\Gamma$ be $n$-$c$-monotone and assume, without the loss of
generality, that $i_0=1$. Let
$(x^1_1,\dots,x_N^1),\dots,(x_1^n,\dots,x_N^n)\in\Gamma$ and $\sigma\in S_n$. Then a straightforward computation implies that the inequality
$$
\sum_{j=1}^n c(x_1^j,x_2^{\sigma(j)},\dots,x_N^{\sigma(j)})\leq \sum_{j=1}^n c(x_1^{j},\dots,x_N^{j})
$$
leads to the inequality
$$
\sum_{j=1}^n\bigg\langle x_1^j,\sum_{i=2}^N x_i^{\sigma(j)}\bigg\rangle\leq \sum_{j=1}^n\bigg\langle x_1^j,\sum_{i=2}^N x_i^j\bigg\rangle.
$$
Thus, we see that if $\Gamma$ is $n$-$c$-monotone, then $A_1$ is $n$-monotone. To sum up, 
\begin{quote}
if for some $1\leq i_0\leq N$ the mapping $A_{i_0}$ is not cyclically monotone, then the set $\Gamma$ is not a $c$-splitting set.
\end{quote}
Indeed, otherwise, $\Gamma$ would have been $c$-cyclically monotone (as we recollected after Definition~\ref{monotonicity definitions}) and, by the above argument, for all $1\leq i_0\leq N$ the mapping $A_{i_o}$ would have been cyclically monotone.

We now address a trivial embedding of all classical maximally monotone operators in the multi-marginal framework. In particular, we obtain maximally $c$-monotone mappings which are not $c$-cyclically monotone. 

\begin{example}\label{trivial embedding}
Let $A:H\rightrightarrows H$ be a maximally monotone mapping. We set $\Gamma\subseteq X$ by
$$
\Gamma=\menge{(x_1,x_2,0,\ldots,0)}{x_2\in Ax_1}.
$$
Then $\Gamma$ is $c$-monotone and we see that $A_1=A$ is maximally monotone. Consequently, by invoking Theorem~\ref{t:MMMM}~\ref{t:MMMM-ii} we conclude that $\Gamma$ is maximally $c$-monotone. In addition, we see that $A$ is $n$-monotone if and only if $\Gamma$ is $n$-$c$-monotone.
Therefore, if $A$ is not $n$-monotone for some $n\geq 3$, then $\Gamma$ is not $n$-$c$-monotone.
Furthermore, since the $n$-$c$-monotonicity of a set is invariant under
shifts, the set $\Gamma=\menge{(x_1,x_2,\rho_3,\ldots,\rho_N)}{x_2\in Ax_1}$
is also maximally monotone for any constant vectors $\rho_3,\ldots,\rho_N\in H$.
\end{example}

 Our next example of a maximally $c$-monotone set which is not a $c$-splitting set does not follow from an embedding of the type in Example~\ref{trivial embedding}.

\begin{example}\label{c-mono not cyclically mono}
	
	Set $N=3$ and for each $1\leq i\leq 3$ set $X_i=\RR^2$. 
	Let $R_\theta$ denote the counterclockwise rotation by the angle $\theta$ in $\RR^2$. Let the set $\Gamma\subseteq X=\big(\RR^2\big)^3$ be defined by
	\begin{equation}\label{non cyclically gamma}
\Gamma=\Big\{\Big(x,\ \tfrac{\sqrt{3}}{2}R_{-\pi/2}x,\ \tfrac{\sqrt{3}}{2}R_{-\pi/2}x\Big) \,\Big|\,x\in\RR^2\Big\}.
\end{equation}
It follows that 
\begin{equation*}
\gr A_1=\Big\{(x,\sqrt{3}R_{-\pi/2}x)\,\Big|\,x\in\RR^2\Big\}
\quad\implies\quad
A_1=\sqrt{3}R_{-\pi/2}.
\end{equation*}
Since $\Gamma=\Big\{\Big(\tfrac{2}{\sqrt{3}}R_{\pi/2}x,x,x\Big) \,\Big|\,x\in\RR^2\Big\}$, we have
\begin{equation*}
\gr A_2=\gr A_3=
\Big\{\Big(x,x+\frac{2}{\sqrt{3}}R_{\pi/2}x\Big)\,\Big|\,x\in\RR^2\Big\}
\quad\implies\quad
A_2=A_3=\sqrt{\frac{7}{3}}R_{\arctan(2/\sqrt{3})}.
\end{equation*}
We see that $A_1$, $A_2$, and $A_3$ are maximally monotone. Consequently, for each $\varnothing\neq K\subsetneq\{1,2,3\}$, the
mapping $A_{K}$ is maximally monotone and it now follows from
Theorem~\ref{t:MMMM} that $\Gamma$ is maximally $c$-monotone in $X$.
Furthermore, since $A_1$ is not $3$-$c$-cyclically monotone, it is not
$c$-cyclically monotone and, consequently, $\Gamma$ is not a $c$-splitting
set. By a straightforward computation, it follows that
\begin{equation*}
J_{A_1}=\frac{1}{2}R_{\pi/3},\ \ 
J_{A_2}=J_{A_3}=\frac{\sqrt{3}}{4}R_{-\pi/6}\ \ \ \ \ \ \text{and}\ \ \ \ \ J_{A_1}+J_{A_2}+J_{A_3}=\Id.
\end{equation*}
Finally, from~\eqref{non cyclically gamma} it is easy to see that $\Gamma_{i,j}$ is monotone for all $1\leq i<j\leq 3$.
\end{example}

We see that in the case $N=3$ the set $\Gamma$ is $c$-monotone if and only if
the mappings $A_1, A_2$ and $A_3$ are monotone. In the following example we
demonstrate that the monotonicity of all of the $A_i$'s no longer implies the
$c$-monotonicity of $\Gamma$ in the case when $N\geq 4$.

\begin{example}\label{e:failed partial sum} 
	In~\cite[Lemma~4.2 and Example~4.3]{BBW1} it was established that: 
	In $X=\RR^2$, 
	let $n\in\{2,3,\ldots\}$, 
	let $\theta \in
	\;\big]\negthinspace\arccos(1/\sqrt{2}),\arccos(1/\sqrt{2n})\big]$,
	set $\alpha=1/(2n\cos(\theta))$, and 
	denote by $R_\theta$ the counterclockwise rotator by $\theta$.
	Then the following hold:
	\begin{enumerate}
		\item $\alpha R_\theta$ and $\alpha R_{-\theta}$ are firmly
		nonexpansive.
		\item\label{ex5.7ii} $n\alpha R_\theta$ and $n\alpha R_{-\theta}$ are not
		firmly nonexpansive.
		\item $n\alpha R_\theta + n\alpha R_{-\theta} = \Id$.
	\end{enumerate}
We employ these facts to construct a set $\Gamma$ as follows. We set $N=2n$ and
$$
T_i=\begin{cases}
\alpha R_\theta, & 1\leq i\leq n;\\
\alpha R_{-\theta}, & n+1\leq i\leq 2n.
\end{cases} 
$$
Define
$$
\Gamma=\menge{(T_1 x,\ldots,T_{2n} x)}{x\in\RR^2}\subseteq X={(\RR^2)}^{2n}.
$$
It then follows that for each $1\leq i\leq N$, the mapping $J_{A_i}=T_i$ is
firmly nonexpansive with full domain. We conclude that the set $\Gamma$
possesses the following properties:
\begin{enumerate}[resume]
	\item for each $1\leq i\leq N$, the mapping $A_i$ is maximally monotone,
	\item $J_{A_1}+\cdots+J_{A_N}=\Id$.
\end{enumerate}
However, due to \ref{ex5.7ii}, the mappings
\begin{equation*}
J_{A_{\{1,\ldots,n\}}}=\sum_{i=1}^n J_{A_i}=\sum_{i=1}^{n}T_i=n\alpha R_{\theta},
\quad\text{and similarly}\quad
J_{A_{\{n+1,\ldots,2n\}}}=n\alpha R_{-\theta}
\end{equation*}
are not firmly nonexpansive, equivalently, $A_{\{1,\ldots,n\}}=R_{-2\theta}$
and $A_{\{n+1,\ldots,2n\}}=R_{2\theta}$ are not monotone. Consequently, by
employing Lemma~\ref{l:c-mono_iff_mono} we conclude that despite the fact
that $\Gamma$ possesses properties (iv) and (v), it is not a $c$-monotone
set.
\end{example}

\begin{remark}

In~\cite{BBW2} the two marginal projections $\Gamma_{i,j}$ of a set
$\Gamma\subseteq X$ were employed, it was established that if the
$\Gamma_{i,j}$'s are cyclically monotone, then $\Gamma$ is $c$-cyclically
monotone and an explicit construction of a $c$-splitting tuple is provided.
However, it was also established that this is a sufficient condition for
$c$-cyclic monotonicity of $\Gamma$ but not a necessary one, in general, as
can be seen in Example~\ref{ex:gamma ij not mono} where we provide a
maximally $c$-cyclically monotone set such that all of its two-marginal projections are
not monotone. In the one dimensional case (i.e., $X_i=\RR$ for each $1\leq
i\leq N$), it was established that $\Gamma$ is $c$-monotone if and only if
all of its two marginal projections $\Gamma_{i,j}$ are monotone. With the
exception of Example~\ref{ex:gamma ij not mono}, in all of our examples of
$c$-monotone sets in this section the set $\Gamma$ had monotone
two-marginal projections $\Gamma_{i,j}$. Thus, a natural question is: 
\emph{How does
the monotonicity and maximal monotonicity of the two-marginal projections
$\Gamma_{i,j}$ relate to the $c$-monotonicity and maximal $c$-monotonicity
of $\Gamma$?} 

\begin{proposition}\label{p:mono two projections}
	Let $X_i=\RR^d$ for each $1\leq i\leq N$. Let $\Gamma\subseteq X$ be a set. Suppose that for each $1\leq i< j\leq N$ the set $\Gamma_{i,j}$ is monotone. Then $\Gamma$ is $c$-monotone.
\end{proposition}

\begin{proof}
The mapping $A_K$ is monotone if and only if for every $(x_1,\ldots,x_N),\ (y_1,\ldots,y_N)\in\Gamma$,
	$$
	0\leq\Scal{\sum_{i\in K}x_i-y_i}{\sum_{j\not\in K}x_j-y_j}.
	$$
	Since the right-hand side is equal to $\sum_{i\in K\atop j\not\in K}\scal{x_i-y_i}{x_j-y_j}$ and since, by the monotonicity of $\Gamma_{i,j}$, $0\leq\scal{x_i-y_i}{x_j-y_j}$, we see that $A_K$ is monotone.  
\end{proof}

To the best of our knowledge, the question whether the maximal monotonicity of the $\Gamma_{i,j}$'s implies the maximal $c$-monotonicity of $\Gamma$ is still open. 

Finally, we note that the maximal $c$-monotonicity of $\Gamma$ does not imply the maximal monotonicity of the $\Gamma_{i,j}$'s even when the $\Gamma_{i,j}$'s are monotone. Indeed, in Example~\ref{trivial
embedding}, we see that although $\Gamma$ is maximally $c$-monotone,
$\Gamma_{i,j}$ is a singleton for all $3\leq i<j\leq N$, thus $\Gamma_{i,j}$ is monotone but not maximally monotone. Even in the case
$N=3$, $\Gamma_{1,3}$ is a proper subset of the graph of the zero mapping
whenever $\Gamma$ is generated by a maximally monotone mapping $A$ without a
full domain. We conclude in this case that $\Gamma$ is maximally
$c$-monotone, however, $\Gamma_{1,3}$ is not maximally monotone.
\end{remark}

\section*{Acknowledgments}
We thank three anonymous referees for their kind and useful remarks.
Sedi Bartz was partially supported by a University of Massachusetts Lowell
startup grant. 
Heinz Bauschke and Xianfu Wang were partially supported by the Natural Sciences and
Engineering Research Council of Canada. Hung Phan was partially supported by
Autodesk, Inc.

\end{document}